\documentclass[12pt]{amsart}
\usepackage{amssymb,verbatim}
\usepackage{graphicx}

\newtheorem{thm}{Theorem}[section]

\newtheorem{lem}[thm]{Lemma}

\theoremstyle{definition}
\newtheorem{defn}[thm]{Definition}
\theoremstyle{remark}

\numberwithin{equation}{section}
\newcommand{\real}{\mathbb{R}}
\newcommand{\complex}{\mathbb{C}}
\newcommand{\sphere}{\mathbb{C}^*}

\newcommand{\disk}{\mathbb{D}}

\newcommand{\Hy}{\mathcal{H}}
\newcommand{\M}{\mathcal{M}}

\newcommand{\hD}{\mathbf{D}}
\newcommand{\hX}{\mathbf{\widehat{X}}}

\newcommand{\hA}{\mathbf{A}}
\newcommand{\haC}{\mathbf{C}}
\newcommand{\ha}{\mathbf{a}}
\newcommand{\hb}{\mathbf{b}}
\newcommand{\hB}{\mathbf{B}}
\newcommand{\hx}{\mathbf{x}}
\newcommand{\hy}{\mathbf{y}}
\newcommand{\hz}{\mathbf{z}}

\newcommand{\hh}{\mathbf{h}}
\newcommand{\hd}{\mathbf{d}}
\newcommand{\hc}{\mathbf{c}}
\newcommand{\hw}{\mathbf{w}}
\newcommand{\hJ}{\mathbf{J}}

\newcommand{\vp}{\varphi}
\newcommand{\ta}{\theta}
\newcommand{\al}{\alpha}
\newcommand{\ba}{\beta}
\newcommand{\ga}{\gamma}
\newcommand{\ol}{\overline}
\newcommand{\0}{\emptyset}
\newcommand{\e}{\varepsilon}
\newcommand{\da}{\delta}
\newcommand{\diam}{\text{diam}}

\newcommand{\bd}{\text{Bd}}
\newcommand{\lam}{\mathcal{L}}
\begin{document}

\title{Extending Isotopies of Planar Continua}
\author[Oversteegen]{Lex~G.~Oversteegen}
\address{University of Alabama at Birmingham, Department of
Mathematics, Birmingham, AL 35294, USA} \email[Lex Oversteegen]
{overstee@math.uab.edu}

\author[Tymchatyn]{E.~D.~Tymchatyn }
\address{University of Saskatchewan, Department of Mathematics and Statistics,
106 Wiggins road, Saskatoon, Canada, S7N 5E6}
\email[Ed Tymchatyn]{tymchat@math.usask.ca}
\thanks{The first named author was supported in part by
NSF-DMS-0405774, and the second named author by NSERC OGP0005616}



\date{\today}




\begin{abstract}In this paper we solve the following problem in the affirmative:
Let $Z$ be a continuum in the plane $\complex$ and suppose that $h:Z\times [0,1]\to\complex$ is an isotopy
starting at the identity. Can $h$ be extended to an isotopy
of the plane? We will provide a  new characterization of an accessible point
in a planar continuum $Z$ and use it to show that an accessible point is preserved during the
isotopy. We show next that the isotopy can be extended over hyperbolic crosscuts.
The proof makes use of the notion of a metric external ray, which mimics the notion
of a conformal external ray,  but is easier to control during an isotopy.
\end{abstract}

\maketitle

\section{Introduction}
Denote the complex plane by $\complex$, the origin by $O$, the open unit disk by $\disk$ and the complex sphere
 by $\sphere=\complex\cup\{\infty\}$.
Suppose that $h:Z\times [0,1]\to \complex$ is an isotopy of a  continuum $Z\subset\complex$
such that if we denote  $h^t=h|_{Z\times\{t\}}$, then
$h^0=id_Z$. We consider the old problem whether the isotopy
$h$ can be extended to an isotopy of the plane.\footnote{We are indebted to Professor
R.~D.~Edwards who communicated this problem to us.}
Under the much more restrictive assumption of a holomorphic motion (where  the parameter
$t$ belongs to the open unit disk and $h$ is holomorphic in $t$) the $\lambda$-Lemma shows that
 $h^t$ can be extended to a quasi-conformal homeomorphism of the entire plane (in this case
  the assumption that $h$ is continuous can
  be relaxed, while the continuity of the extension still follows,
 see \cite{manesadsull83,lyub83a,sullthur86} and \cite{slod91}
for further details). Although the $\lambda$-Lemma also holds for arbitrary
(in particular not connected) sets $Z$, easy examples show that an isotopy of a convergent sequence cannot necessarily
be extended over the plane (see \cite[p.~991]{fabe05}). It follows from Rado's Theorem \cite[Theorem 4.2]{wen91}
 that the isotopy $h^t$ can  be extended to an
isotopy of $\complex$ if $Z$ is a simple closed curve (see \cite{baer27,baer28} for related results and
\cite{epst66} for a generalization). Analytic techniques, in particular, boundary values of conformal
maps have been  powerful tools for studying
plane continua. However, they appear insufficient to answer the general question.

One of the main complications addressed in this paper is that Carath\'eodory kernel convergence
is insufficient to allow us control of the behavior of conformal external rays under an isotopy. To this end we introduce metric
external rays which depend only on  distance and, hence, behave well under an isotopy. The existence of metric
 external rays was alluded to in \cite{bell67} and more fully developed in \cite{ilia70,bell76}.
 For our purpose it will be easier to define them as the equidistant set between two disjoint
 and closed sets in the covering space of $\complex\setminus \{O\}$ by the exponential map. Equidistant sets
 and metric external rays were
 studied in detail by G. Brouwer in \cite{brou05}.
 We will use  metric external rays to show that the isotopy can be extended
 over  conformal external rays. Even though the entire proof could have been
 carried out using  metric external rays, it suffices for this paper to rely for the final extension over $\complex^*$
 on the existing analytic theory.

 We will always denote by   $Z$ a proper subcontinuum in the sphere $\sphere$ (or equivalently in the plane $\complex$),
 by $h:Z\times [0,1]\to\complex$ an isotopy
  such that $h^0=id_Z$ and  by $U$  a component of $\sphere\setminus Z$ (or equivalently $\complex\setminus Z$).
Given a fixed component $U$ of $\sphere\setminus Z$ we may assume, without loss of generality,
  that $U$  contains the point at infinity
(or is the unbounded component of $Z$)
and  $\infty\in\sphere\setminus h^t(Z)$ for all $t\in[0,1]$.
Denote by $U^t$  the  component of $\sphere\setminus h^t(Z)$
containing the point at infinity  (or the unbounded component of $\complex\setminus h^t(Z)$),
then $U^t\cup\{\infty\}$ is  simply connected.
We always denote by  $\vp^t:\disk\to U^t\cup \{\infty\}$  the conformal map such that
$\vp^t(O)=\infty$ and $(\vp^t)'(O)>0$. Then the maps $\vp^t$ are unique and,
by Carath\'eodory kernel convergence, uniformly convergent in $t$ on compact subsets of $\disk$.
By slightly abusing the language we will identify points in the boundary $S^1$ of the disk $\disk$
with their arguments and call them \emph{angles}.

We say  that $x\in Z$ is  \emph{accessible}  from $U$
 if there exists an angle $\ta\in [0,2\pi)$ such that
 the \emph{(conformal) external ray} $R_\ta=\vp(\{r e^{ i \theta} \mid r<1\})$ lands on $x$
(i.e., $\ol{R_\ta}\setminus R_\ta=\{x\}$).  It is well-known that  a point $x\in Z$
 is accessible from $U$ if and only if there exists a continuum $Y\subset\ol{U}$ such that $Y\cap Z=\{x\}$.
 Moreover, in this case $\ol{\vp^{-1}(Y\setminus \{x\})}\cap S^1=\{\ta\}$ is a single point and $R_{\ta}$ lands
 on $x$ in $Z$ \cite{miln00}. It is clearly necessary that the corresponding point
$x^t=h^t(x)$ remains  accessible  in $h^t(Z)$ from $U^t$. However, Carath\'eodory kernel convergence is insufficient
to show this and one of the first steps of the proof is to
show that this is indeed the case. If we assume  in addition that $x$ is not a  cut point of $Z$,
then there exists for each $t$ a unique angle $\theta^t$ such that the external ray $R^t_{\ta^{t}}$
of $Z^t$ lands on
$x^t$.

The next step of the proof is to show that this correspondence of angles is continuous in $t$ and
 there exists an isotopy $\al^t:S^1\to S^1$ of the unit circle such that
 if $R^0_\ta$ lands on $x^0$ in $Z^0$, then
 $R^t_{\al^t(\ta)}$ lands on $x^t$ in $Z^t$ for each $t$.   Extending $\al^t$ to an isotopy
 $f^t:\disk \to \disk$, defined by $f^t(re^{i\ta})=re^{i\al^{t}(\ta)}$ does not, however, provide a
 proper  extension over $\ol{U^0}$
 since simple examples show that in general the isotopy $H:U^0\times [0,1]\to\sphere$ defined by
 $H(w,t)=\vp^t\circ f^t\circ (\vp^0)^{-1}(w)$ does not have a continuous extension
 over $\bd(U^0)$. In the final step of the proof we use the Kulkarni-Pinkall \cite{kulkpink94}
 lamination of $U^0$ in the sphere to define the proper extension.

We denote  by $\exp$ the covering map
 $\exp:\complex\to\complex\setminus \{O\}$ defined by $\exp(z)=e^{z}$.
 Given a set $X\subset \complex$ we denote by $\hX=\exp^{-1}(X\setminus \{O\})$
 and we use bold face letters for subsets of $\hX$.
 However, for points $x\in\complex\setminus \{O\}$
  we denote by $\hx$ a point in the set
 $\exp^{-1}(x)$. We also denote
  by $\pi_j:\complex\to\real$, $j=1,2$,  the projections onto the $x$-axis and $y$-axis, respectively.
  The open ball with center $x$ and radius $r$ is denoted by $B(x,r)$ and its boundary by $S(x,r)$.
  For a set $A\subset\complex$ we denote by $B(A,\e)=\bigcup\{B(a,\e)\mid a\in A\}$.
  By a \emph{ray} $R$ we mean a  subset
  of $\complex$ homeomorphic to the real line $\real$ so that $|\ol{R}\setminus R|\le 1$
  and $\ol{R}$ is not a simple closed curve. If $\ol{R}\setminus R=\0$, then we say that $R$ is a
  \emph{closed ray}.

  We will use the following notation throughout:
   for any set $A\subset Z$ we denote by $A^t$ the set $h^t(A)$. We are initially
   only interested in extending the isotopy over the unbounded component $U$
   of $\complex\setminus Z$.
   Recall that $U^t$ is the component of $\sphere\setminus
   h^t(Z)$ containing $\infty$ and denote by $X^t$ the continuum $\sphere\setminus U^t$. Then $X^t$ is a non-separating
   plane continuum. We may identify any particular point $z\in\bd(U)$
   with the origin $O$, assume that it is fixed under the isotopy and that $X^t\subset B(O,1)$
   for all $t\in[0,1]$. We will denote the Euclidean metric on $\complex$ by $d$ and the spherical metric
   on $\sphere$ by $\rho$.
  Finally, given two points $x,y\in\complex$, we denote by $xy$ the straight line segment joining them.

\section{Preliminaries}\label{Sprelim}
Crucial to our study is the notion of an equidistant set between two disjoint closed sets in $\complex$.
Since the reference \cite{brou05} is not easily accessible and our results require a slightly different setting,
 we will sketch proofs for some of these results in this section.
 We start with the following definition from \cite{brou05}.
Suppose that $A$ and $B$ are two disjoint closed subsets of the plane. For $z\in\complex\setminus[A\cup B]$,
let $r(z)=d(z,A\cup B)$.
Then we say that $A$ and $B$ are \emph{non-interlaced} if for each $z\in \complex\setminus [A\cup B]$,
 $A\cap S(z,r(z))$ and $B\cap S(z,r(z))$
 are contained in two disjoint closed and connected subsets of $S(z,r(z))$ (one may be empty).
  Let $E(A,B)=\{z\in \complex \mid
 d(z,A)=d(z,B)\}$ be the equidistant set between $A$ and $B$.

 Let $A$ and $B$ be two disjoint, closed and non-interlaced sets. By Gaston Brouwer \cite{brou05}[Theorem 3.4.4],
  $E(A,B)$ is a $1$-manifold.
Moreover, if $A$ and $B$ are connected, then $E(A,B)$ is connected and, hence, it is either a closed ray in the plane
or a simple closed curve. In particular if $A$ and $B$ are also both unbounded, then $E(A,B)$ is a closed ray
which separates $\complex$ into two disjoint open and connected sets one containing $A$ and the other containing $B$.
We will slightly generalize this case by replacing the condition that $A$ and $B$ are connected by
the weaker condition that  $A$ lies above  $B$ (see Definition~\ref{dorder} and Theorem~\ref{connectedE}).

Since  $O\in X^t\subset B(O,1)$ for all $t$,
 $\max\{\pi_1(\hX^t)\}<0$.
 It follows from this and the fact that $X$ is a continuum
 that for any component $\haC$ of $\hX$,
$\pi_1(\haC)=(-\infty,m_{\haC}]$ with $m_{\haC}<0$. Moreover, since $X$ is non-separating,  $\hX$ is also
non-separating and, hence, each component $\haC$ of $\hX$ is also non-separating.
To see this note that $\hX$ has a unique complementary domain
 $W$ such that $\pi_1^{-1}([0,\infty))\subset W$. If $V$ is any complementary domain of $\hX$, then $V$ must contain
  a point $\mathbf{v}\in V\setminus \hX$ and $\exp(\mathbf{v})=v\in\complex\setminus X$. Hence there exists a ray
  $R\subset\complex\setminus X$ joining $v$ to infinity. Then the lift $\bf R$ of $R$ with initial point $\mathbf{v}$
  is a ray in $\complex\setminus \hX$ which intersects $W$ and $V=W$ as required. These facts will allow
 us to define what it means for one component of $\hX$ to lie above another component.

\begin{defn}\label{dorder}Let $\haC$ and $\hD$ be two distinct components of $\hX$. We say that $\haC$ \emph{lies above}
 $\hD$
if there is a path $s:[0,1]\to \pi_1^{-1}((-\infty,0])\setminus \haC$  such that the initial point $s(0)$ is in $\hD$,
$s(1)=O$ and
if $R=s([0,1])\cup [0,\infty)\times \{0\}$, then $\haC$ lies in the unique unbounded component of
 $\complex\setminus [\hD\cup R]$ which contains the point $1+2\pi i$.

 Moreover, if $\hX=\hA\cup \hB$, where $\hA$ and $\hB$ are disjoint closed sets,
  such that every component of $\hA$ lies above every component of $\hB$, then
 we say that $\hA$ \emph{lies above} $\hB$.
\end{defn}

\begin{lem} The notion of $\haC$ lying above $\hD$ is well-defined and for any two components $\haC$ and $\hD$ of $\hX$,
one must lie above the other.
\end{lem}
\begin{proof} Let $s_1$ and $s_2$ be two
paths  satisfying the conditions in
definition~\ref{dorder}. Put $R_1=s_1([0,1])\cup [0,\infty)\times \{0\}$,
$R_2=s_2([0,1])\cup [0,\infty)\times \{0\}$ and suppose that $\haC$ lies in the unbounded component
of $\hD\cup R_1$ which contains the point
$1+2\pi i$.
  Suppose first that $R_1$ and $R_2$ have the same initial point
  $s_1(0)=s_2(0)$.
Since $\pi_1^{-1}((-\infty,0])\setminus \haC$ is simply connected, there exists a homotopy
 $j:[0,1]\times[0,1]\to \pi_1^{-1}((-\infty,0])\setminus \haC$,
with endpoints fixed, between $s_1$ and $s_2$. Since $j^t$ misses the connected set $\haC$ for each $t$,
it follows that $\haC$ lies in the  component of $\complex\setminus [R_2\cup \hD]$ containing $1+2\pi i$.

Next suppose that $R_1$ and $R_2$ have initial points $\hz_1$ and $\hz_2$, respectively. Let $\mathcal{U}=
\{B(y,\e) \mid y\in \hD \text{ and } \e=(1/3)\, d(y,[C\cup \pi_1^{-1}([0,\infty))])\}$.
Then $\mathcal{U}$ is an open cover of $\hD$.
Since $\hD$ is connected, there exists a chain $\{B_1,\dots,B_n\}$ of balls in $\mathcal{U}$ such that $\hz_1\in B_1$,
$\hz_2\in B_n$ and $B_j\cap B_{j+1}\not=\0$ for $j=1,\dots,n-1$. Let $J$ be a piecewise linear arc in
$\cup B_j$ from $\hz_1$
 to $\hz_2$. Then there exists a path $s_3$ such that $s_3([0,1]))=J \cup s_2([0,1])$
  is a path with initial
 point $\hz_1$ and terminal point $O$. Put $R_3=s_3([0,1])\cup  [0,\infty)\times \{0\}$.
 Then $\haC$ lies in the unbounded component of $\complex\setminus [\hD\cup R_3]$ which contains the point
 $1+2\pi i$. Hence, $\haC$ lies in the unbounded component of $\complex\setminus [\hD\cup R_2]$ which contains
 the point $1+2\pi i$.

 Suppose that $\haC$ and $\hD$ are any two components of $\hX$. Then $U_{\haC}=\complex\setminus \haC$ and
 $U_\hD=\complex\setminus \hD$ are open and connected sets homeomorphic to $\complex$.
 Hence there exist two arcs $J_\haC\subset
 U_\hD\cap \pi_1^{-1}((-\infty,0])$ and $J_\hD\subset U_\haC\cap \pi_1^{-1}((-\infty,0])$ joining points $\hc\in \haC$
 and $\hd\in \hD$ to $O$, respectively.
 In addition we may assume that $J_\haC\cap J_\hD=\{O\}$.
 If $\hD$ is not contained in the component of $\complex\setminus J_\haC\cup [0,\infty)\times \{0\}$ containing $1+2\pi i$,
  then $\haC$ is contained in the component of $\complex\setminus J_\hD\cup [0,\infty)\times \{0\}$ containing
  $1+2\pi i$ and $\haC$ lies above $\hD$.
\end{proof}

Our goal is to show that the condition that $\hA$ lies above $\hB$ is preserved under the lift of
the isotopy $h^t$.

 \begin{lem}\label{lifth} Suppose $h^t:\bd(X)\to\complex$ is an isotopy such that $h^0=id_{\bd(X)}$,
 $O\in \bd(X)$ and $h^t(O)=O$ for all $t$.
 Then there exist an isotopy $\hh^t:\bd(\hX)\to \complex$  which lifts $h^t$ such that $\hh^0=id_{\bd(\hX)}$.
 \end{lem}

 \begin{proof}
For each $x\in \bd(X)\setminus \{O\}$ and each $\hx\in\exp^{-1}(x)$ the path $h|_{\{x\}\times [0,1]}$ has a unique lift
to a path $\hh_{\hx}:[0,1]\to\complex$ with initial point $\hx$. Define $\hh^t(\hx)=\hh_{\hx}(t)$.
By uniqueness of lifts, $\hh^t$ is one-to-one. It now follows easily that $\hh^t$ is an isotopy
of $\bd(\hX)$ lifting $h^t$ with $\hh^0=id_{\bd(\hX)}$.
\end{proof}

The following easy Lemma follows immediately from the fact that $h^t(O)=O$ for all $t$ and that $h^t$
 is uniformly continuous.
\begin{lem}\label{left} Suppose that $h^t(O)=O$ for all $t$ and let $\hh^t:\bd(\hX)\to\complex$ be the isotopy
which is the lift of $h^t$ to $\bd(\hX)=\exp^{-1}(\bd(X)\setminus\{O\})$ such that $\hh^0=id_{\bd(\hX)}$.
 Denote $\hh^t(\hx)$ by $\hx^t$.
For all $\e>0$ there exists $\da\in \real$ such that if there exists $t_0\in[0,1]$ such
that $\hx^{t_0}\in \hX^{t_0}$ and
    $\pi_1(\hx^{t_0})\le \da$, then $\pi_1(\hx^t)<\e$ for all $t\in[0,1]$.
    In otherwords, if there exists
    $t_0$ such that $\pi_1(\hx^{t_0})\ge \e$, then $\pi_1(\hx^t)>\da$ for all $t\in[0,1]$.
\end{lem}

Given the existence of the lifted isotopy $\hh^t$ we will use  similar notation as for $h^t$:
for any set $\hA\subset \bd(\hX)$ we denote by $\hA^t$ the set $\hh^t(\hA)$. Recall that
$U^t$ is the unbounded component of $\complex\setminus h^t(\bd(X))$, $X^t=\complex\setminus U^t$
and $\hX^t=\exp^{-1}(X^t\setminus\{O\})$.
Also, if $\haC^0$ is a component
of $\hX^0$ choose a point $\hx^0\in \bd(\hX^0)\cap \haC$.  Then we denote by $\haC^t$ the component of $\hX^t$
containing the point $\hx^t=\hh^t(\hx)$.
Next we show that the notion of the component $\haC$ being above $\hD$ in $\hX$  is preserved throughout the isotopy
$\hh$.

\begin{lem} \label{order} Let $\haC=\haC^0$ and $\hD=\hD^0$ be components of $\hX^0$ such that $\haC$ lies above $\hD$.
Then $\haC^t$ lies above $\hD^t$ for each $t\in [0,1]$.
\end{lem}

\begin{proof}It suffices to show that there exists $0<t_0$ such that for all $t\le t_0$
$\haC^t$ lies above $\hD^t$.  Let $R=s([0,1])\cup [0,\infty)\times\{0\}$ be a  piecewise linear  ray
landing on $\hd^0\in \hD^0$
which satisfies the conditions
of Definition~\ref{dorder} and such that $R\cap \haC^0=\0$ and $R\cap \hD^0=\{\hd^0\}$.
 Then $R\cup \hD^0$ has exactly two complementary domains
and each is homeomorphic to $\complex$. Hence there exists an arc $A\subset \complex\setminus [\hD^0\cup R]$
 joining a point $\hc^0\in \haC^0$ to the point $1+2\pi i$. Choose $a<0$ such that
  $A\cup R\subset \pi^{-1}([a,\infty))$. Choose $\e<(1/3)\ d(A\cup [\pi_1^{-1}([2a,\infty))\cap \haC^0],
  R\cup [\pi_1^{-1}([2a,\infty))\cap \hD^0])$.
 Let $0<t_0$ such that for each  $\hx\in \bd(\hX)\cap \pi_1^{-1}([2a,\infty))$, $|\hh^t(\hx)-\hh^0(\hx)|<\e/2$
  and
 $\pi_1(\hh^t(\hx))<a$ for all $\hx\in \pi_1^{-1}((-\infty,2a])\cap \bd(\hX)$  for all $t$.
  Then for all $t\le t_0$,
 $\haC^t\cup \hc^0 \hc^t\cup A$ is connected, closed and disjoint from $\hD^t\cup \hd^0 \hd^t\cup R$.
 The first set contains an arc
 $A^*$ from $\hc^t$ to the point $1+2\pi i$ and the latter set
 contains a half ray $R^*$ satisfying the conditions in ~\ref{dorder} from the point $\hd^t\in \hD^t$ to $\infty$.
 Since $1+2\pi i$ is above $R^*$ and $A^*\cap [\hD^t\cup R^*]=\0$, $\haC^t$ is above $\hD^t$ for all $t\in[0,t_0]$.
 \end{proof}

\begin{lem} \label{non-interlaced} Suppose that $\hX^0=\hA^0\cup \hB^0$
are disjoint closed sets such that  $\hA^0$
lies above  $\hB^0$. Then for each $t$,
$\hX^t=\hA^t\cup \hB^t$ and $\hA^t$ and $\hB^t$ are disjoint, closed and non-interlaced sets.
\end{lem}

\begin{proof} By Lemma~\ref{order}, every component of $\hA^t$ lies above every component of $\hB^t$ for all $t$.
Since $\hh^t$ is an isotopy it only remains to   show that $\hA^t$ and $\hB^t$ are non-interlaced.
To see this fix $t$, let $\hw\in E(\hA^t,\hB^t)$ and let $K\subset \complex \setminus \hA^t\cup \hB^t$
 be the minimal open ball with center
$\hw$ whose boundary $S$ meets both $\hA^t$ and $\hB^t$. Suppose that there exist $\hx,\hx'\in S\cap \hA^t$ and
$\hy,\hy'\in S\cap \hB^t$ such that $\{\hy,\hy'\}$ separates $\hx$ and $\hx'$ in $S$. For $\hz\in \hX^t$,
let $\haC_\hz$ denote the component of
$\hX^t$ which contains the point $\hz$. Suppose, without loss of generality, that $\haC_\hy$ lies above $ \haC_{\hy'}$.
We may suppose that  $\haC_\hy\cup \haC_{\hy'}\cup \hy\hy'$ separates $\hx$ from $1+2\pi i$ in
 $\complex$. Since $\haC_\hy\cup \haC_{\hy'}$ does not separate $\complex$ by unicoherence, we can choose an arc
 $D$ in $\complex\setminus[\haC_\hy\cup \haC_{\hy'}]$ irreducible from $O$ to $\hy\hy'$ such that $\pi_1(D)\subset (-\infty,0]$.
 Let $\{d\}=D\cap \hy\hy'$. Then $\haC_\hy\cup \hy d\cup D\cup [0,\infty)\times\{0\}$ separates $\hy'$,
  and hence also $\hx$,
 from $1+2\pi i$ and $\haC_\hx$ is below $\haC_\hy$. This contradicts Lemma~\ref{order} and completes the proof.
\end{proof}

\begin{lem} \label{oneside}Suppose $\hX=\hA\cup \hB$, where
$\hA$ and $\hB$ are disjoint closed   subsets of $\complex$ such that $\hA$ lies above $\hB$.
Let $E$ be a component of $E(\hA,\hB)$. Then $E$ is a closed ray.
If $e\in E$ and $r=d(e,\hA\cup \hB)$, then there exist disjoint irreducible arcs or points
$J_\hA$ and $J_\hB$ in $S(e,r)$ such that $\hA\cap S(e,r)\subset J_\hA$ and $\hB\cap S(e,r)\subset J_\hB$,
and $E$ separates $J_\hA$ from $J_\hB$ in $\complex$.

\end{lem}
\begin{proof} By Lemma~\ref{non-interlaced}, $\hA$ and $\hB$ are non-interlaced.
By \cite{brou05}[Theorem 3.4.4], $E$ is a $1$-manifold. Let $E$ be a component of $E(\hA,\hB)$,
$e\in E$ and $d(e,\hA\cup \hB)=r$. Since $\hA$ and $\hB$ are non-interlaced,
 there exist disjoint irreducible arcs or points
$J_\hA$ and $J_\hB$ in $S(e,r)$ such that $\hA\cap S(e,r)\subset J_\hA$ and $\hB\cap S(e,r)\subset J_\hB$.
Let $\ha_1$ and $\ha_2$
be the endpoints of $J_\hA$. For $\hz\in\hX$, let $\haC_\hz$ be the component of $\hz$ in $\hX$. Let $V$
be the component of $\complex\setminus [\haC_{\ha_1}\cup J_\hA\cup \haC_{\ha_2}]$ containing $e$
 and let $W=\complex\setminus V$.
It follows from the proof of Lemma~\ref{non-interlaced} that $W\cap \hB=\0$.

Choose $\hz\in J_\hA\setminus \hA$ and $\hw\in \hB$. If $\hz \hw\cap \hA\ne\0$, then $d(\hz,\hA)<d(\hz,\hw)$.
If $\hz \hw\cap \hA=\0$, then it follows easily
that $d(\hz,\hw)>\min\{d(\hz,\ha_1),d(\hz,\ha_2)\}$. Hence for all $\hw\in \hB$,
$d(\hz,\hA)<d(\hz,\hw)$ and $E(\hA,\hB)\cap J_\hA=\0$.
Choose $\hb\in J_\hB\cap \hB$.
Note that $E(\hA,\hB)\cap \ha_1e\setminus \{e\}=\0=E(\hA,\hB)\cap e\hb\setminus\{e\}$. Now $E(\hA,\hB)$
separates $\ha_1$ and $\hb$.
By unicoherence of $\complex$ a component of $E(\hA,\hB)$ separates $\ha_1$ and $\hb$.
 Since this component must contain $e$,
$E$ separates $\ha_1$ and $\hb$ in $\complex$. Hence $E$ separates $\haC_{\ha_1}\cup J_\hA$ and $\haC_\hb\cup J_\hB$
which both are unbounded sets. Hence, $E$ is an unbounded closed $1$-manifold, i.e., $E$ is a closed ray.
\end{proof}

\begin{thm}\label{compactV}\label{connectedE} Suppose that $\hX=\hA\cup \hB$,
where $\hA$ and $\hB$ are disjoint, closed, non-empty sets such that $\hA$ lies above $\hB$.
Then $E(\hA,\hB)$ is a closed ray such that $\pi_1(E(\hA,\hB))=(-\infty,\infty)$ and for
$x>0$, $|\pi_1^{-1}(x)\cap E(\hA,\hB)|=1$.
\end{thm}
\begin{proof} By Lemma~\ref{oneside}, each component of $E(\hA,\hB)$ is a closed ray which stretches to $-\infty$.
For $\hz\in\hX$, let $\haC_\hz$ be the component of $\hz$ in $\hX$.

Let $\ha\in \hX$ such that $\pi_1(\ha)=\max(\pi_1(\hX))<0$. Without loss of generality, $\ha\in \hA$.
Let $R=\ha O\cup [0,\infty)\times\{0\}$, then $R\setminus \{\ha\}$ is a ray disjoint from $\hX$ which lands on $\ha$.
Note that $\complex\setminus [R\cup \haC_\ha]=W\cup V$, where $W$ and $V$ are disjoint, connected, open
and non-empty sets. Without loss of generality, $1+2\pi i\in W$. Then every component of $\hX\cap W$ is above
$\haC_\ha\subset \hA$. Hence $\hB\subset V$. Since $\hB\ne\0$, there exists $\hb\in \hB$ such that $\pi_1(\hb)=\pi_1(\ha)$.
By compactness of $X\cap S(O,e^{\pi_1(\ha)})$, we may assume that
\begin{gather*}\pi_2(\ha)=\min(\pi_2(\hA\cap \pi_1^{-1}(\pi_1(\ha))))  \text{ and}\\
\pi_2(\hb)=\max(\pi_2(\hB\cap\pi_1^{-1}(\pi_1(\ha)))).
\end{gather*}
Then $0<\pi_2(\ha)-\pi_2(\hb)\le 2\pi $ and we may assume that $0<\pi_2(\ha)\le\pi$. For
$z\in [\pi_1(\ha),\infty)\times [\pi_2(\ha),\infty)$,
$d(z,\hA)<d(z,\hB)$ and for $z\in [\pi_1(\ha),\infty)\times (-\infty,\pi_2(\hb)]$, $d(z,\hB)<d(z,\hA)$.
By unicoherence there exists a component $E$ of $E(\hA,\hB)$ which separates $\ha$ and $\hb$. Then
$E$ separates $\haC_\ha\cup [\pi_1(\ha),\infty)\times\pi_2(\ha)$ from $\haC_b\cup [\pi_1(\ha),\infty)\times \pi_2(\hb)$.
So $\pi_1(E)=(-\infty,\infty)$ and
  $E(\hA,\hB)\cap [\pi_1(\ha),\infty)\times \real]\subset [\pi_1(\ha),\infty)\times [\pi_2(\hb),\pi_2(\ha)]$.
  In particular, $E(\hA,\hB)\cap \pi_1^{-1}(x)$ is compact for $x>0$.

  It remains to be shown that $E(\hA,\hB)$ is connected. We prove first that $E$ separates $\hA$ and $\hB$.
  Suppose that $\complex\setminus E=W'\cup V'$, where $W'$ and $V'$ are disjoint, non-empty, open and connected sets,
  and $1+2\pi i\in W'$. Just suppose there exist $\hy\in \hA\cap V'$. Since neither of the disjoint closed sets
  $E$ nor $\hB$ separates $\hy$ from $1-2\pi i$, neither does their union. Let $D\subset \complex\setminus[E\cup B]$
  be an arc
  from $\hy$ to $1-2\pi i\in V'$.  Choose $e\in E$ such that if $r=d(e,\hX)$, then $\pi_1(e)+r<\min\{\pi_1(D)\}$.
  Let $\hw\in S(e,r)\cap \hB$. Then $\haC_\hw\cup \hw e\cup E'$, where $E'$ is the component of $E\setminus\{e\}$
  which projects under $\pi_1$ over $[\pi_1(y),\infty)$, does not separate $\hy$ from $1-2\pi i$.
  It now follows easily that $\haC_\hy$ lies below $\haC_\hw$, a contradiction.

  Hence, we can conclude that $\hA\subset W'$, $\hB\subset V'$ and $E(\hA,\hB)\cap [W'\cup V']=\0$.
  It follows that for all $z\in W'$, $d(z,\hA)<d(z,\hB)$ and for all $z\in V'$, $d(z,\hB)<d(z,\hA)$, and $E(\hA,\hB)=E$.

  Suppose that $e_1\ne e_2\in E$ and  $d(e_i,\hA\cup \hB)=r_i$. Then for all
  $\hz_i\in S(e_i,r_i)\cap [\hA\cup \hB]$, $[e_1\hz_1\setminus \{\hz_1\}]\cap [e_2\hz_2\setminus \{\hz_2\}]=\0$.
Since $\hA$ lies above $\hB$,
   for $e\in E$ with $\pi_1(e)>0$ and $r=d(e,\hA\cup \hB)$, for all $\hz\in \hA\cap S(e,r)$ and $\hw\in \hB\cap S(e,r)$,
   $\pi_2(\hz)>\pi_2(\hw)$. It now follows easily that for such $e$, $\pi^{-1}_1(\pi_1(e))\cap E(\hA,\hB)=\{e\}$.

\end{proof}

 \section{Characterizing accessibility}\label{Schar}
In this section we provide a characterization of accessibility for points in $X$ and show that
accessibility is preserved under the
isotopy $h$. In this section we will always assume that $O\in \bd(X)$ is fixed under the isotopy $h$.
Easy examples (e.g., a half ray spiralling around the closed interval $[-1,1]\times \{0\}$) show that
accessibility of $O$ in $X$ is not equivalent to $\hX$ being not connected.
Nevertheless the spirit of this idea is correct:

\begin{lem} \label{char} Suppose that $O\in \bd(X)$. Then the following are equivalent:
\begin{enumerate}
\item $O$ is accessible,
\item  \label{charsep}$\hX=\hA\cup \hB$, where $\hA$ and $\hB$ are non-empty, disjoint and closed
such that:
\[\text{$\hA$ lies above  $\hB$}\]  and
for all $x\in \real$ there exists $y_1<y_2$
in $\real$ such that
\begin{gather*}
\pi_1^{-1}([x,\infty))\cap \pi_2^{-1}([y_2,\infty))\cap \hB=\0 \text{ and } \tag{2a}\label{2a}\\
\pi_1^{-1}([x,\infty))\cap \pi_2^{-1}((-\infty,y_1])\cap \hA=\0 \tag{2b}\label{2b}.
\end{gather*}
\end{enumerate}
\end{lem}
\begin{proof}
Suppose first that $O$ is accessible, let $R$ be a conformal external ray landing on $O$
and let $\hJ$ be a component of $\exp^{-1}(R)$. Then $\hJ$ is a closed ray in $\complex\setminus \hX$ such that
$\pi_1(\hJ)=(-\infty,\infty)$ and for every vertical line $\ell$,  $\hJ\cap \ell$ is compact.
Note that  $\complex\setminus \hJ=U \cup V$, where $U$ and $V$ are disjoint open and connected sets.
We may assume that for some vertical line $\ell$,
 $\pi_2(\ell\cap U)$ has no upper bound and, since $\hX$ is invariant under vertical translations by $2\pi$,
 that $\{1+2\pi i\}\subset U$. Put $\hA=\hX\cap U$ and $\hB=\hX\cap V$
then condition (\ref{charsep}) holds. The fact every component of $\hA$ lies above
every component of $\hB$ follows from the fact
 that $U$ and $V$ are open and connected, and $U$ is \lq\lq above\rq\rq $V$. To see that (\ref{2a}) and (\ref{2b}) hold
  note that close to infinity $R$ behaves like a radial line segment
in the plane and, hence, $\hJ$ behaves like a horizontal line segment near $+\infty$
so $\pi_2(\hJ\cap \pi_1^{-1}([a,\infty)))$ is bounded for each $a\in\real$.

Suppose next that conditions~(\ref{charsep}), (2a) and (2b) hold. By Theorem~\ref{connectedE},
$E=E(\hA,\hB)$ is a closed ray which runs from $-\infty$ to $\infty$ and separates $\hA$ from $\hB$.

{\bf Claim}: For every  compact arc $C\subset x$-axis,  $\pi_1^{-1}(C)\cap E$ is compact.

Proof of Claim. Let $C\subset x$-axis be a compact interval.
Suppose without loss of generality that  $\pi_1^{-1}(C)\cap E$ contains points $e_i$
with $\lim \pi_2(e_i)=+\infty$. Note that there exists $K>0$ such that for each $z\in \pi_1^{-1}(C)$,
 $d(z,\hX)\le K$.
Hence for each $i$ there exists $\hb_i\in \hB$ such that $d(e_i,\hb_i)\le K$ and $\lim \pi_2(\hb_i)=+\infty$.
This contradicts (\ref{2a}) and completes the proof of the claim.

Now let $MR=\ol{\exp(E)}$. Then $MR$ is a closed and connected set  in $\complex$ and it suffices to show that
$MR\cap X=\{O\}$. Since $\pi_1(E)=(-\infty,\infty)$, $O\in MR$. Suppose that $x\in X\setminus \{O\}$
 is also a limit point of
$MR$. Choose $z_i\in \exp(E)$  such that $\lim z_i=x$ and $\hz_i\in E$ such that $\exp(\hz_i)=z_i$.
Then $d(\hz_i,\exp^{-1}(x))\to 0$. Since $E$ is closed and disjoint from $\hX$,
the sequence $\hz_i$ cannot be convergent and so $\lim |\pi_2(\hz_i)|=\infty$.
By choosing a compact arc $C$ in $R$ which contains $\pi_1(\exp^{-1}(x))$ in its interior, we see that
$E\cap \pi_1^{-1}(C)$ is not compact. This contradiction completes the proof.
\end{proof}

The characterization Lemma~\ref{char} allows us to
 show that an accessible point remains accessible throughout the isotopy.

\begin{thm}\label{ptpreserved}If $x$ is  accessible from $U=\complex\setminus X$, then
$h^t(x)$ is  accessible from $U^t$, where $U^t$ is the unbounded component of
 $\complex\setminus h^t(\bd(X))$, for each $t\in[0,1]$.
\end{thm}

\begin{proof} Suppose that $x^0$ is an accessible point of $X^0$. We may assume that $x^0=O$,
   $h^t(O)=O$ and $X^t\subset B(O,1)$ for all $t$.

  By Lemma~\ref{char} $\hX^0=\hA^0\cup \hB^0$ such that  conditions  (\ref{charsep}), (2a)
  and (2b) of
   Lemma~\ref{char} hold.
  By Lemma~\ref{lifth} we can lift the isotopy $h^t$ to an isotopy $\hh^t:\bd(\hX^0)\to\complex$ such that
  $\hh^0=id_{\bd(\hX^0)}$. By Lemma~\ref{order},  $\hA^t$ lies above  $\hB^t$ for all $t$.
  It remains to be shown that conditions (\ref{2a}) and (\ref{2b}) are satisfied for all $t$.
  By symmetry it suffices to show that (\ref{2a}) holds.
    Suppose $x\in \real$. By Lemma~\ref{left} there exists $x'\in \real$
  such that for all $t$, $\max(\pi_1(\hh^t(\pi_1^{-1}({(-\infty,x'])}\cap \hX^0))<x$.
  By~(\ref{2a})  for $t=0$, there exists $y_2$ such that $\pi_1^{-1}({[x',\infty))}\cap \pi_2^{-1}([y_2,\infty))
  \cap B^0=\0$. Choose $y_3$
  such that for all $t$, $\max(\pi_2(\hh^t(\pi_1^{-1}([x',\infty))\cap \pi_2^{-1}((-\infty, y_2]))\cap \hX)))<y_3$.
  Then $\pi_1^{-1}([x,\infty))\cap \pi_2^{-1}([y_3,\infty))\cap B^t=\0$ and (\ref{2a}) holds for all $t$.
  Hence by Lemma~\ref{char}, $O$ is accessible for all $t$.
\end{proof}

\section{Continuity of external angles}  Given a non-separating continuum $X$, an isotopy
$h^t:\bd(X)\to \complex$ such that $h^0=id_{\bd(X)}$, let $U^t$ be the unbounded component of
$\complex\setminus h^t(\bd(X))$.
We construct an isotopy $\al^t:S^1\to S^1$ such that if the conformal ray $R_\ta\subset U^0$ lands on $x$, then
$R_{\al^t(\ta)}\subset U^t$ lands on $x^t$ in $X^t$ for each
$t$.  This is accomplished in two steps. We first construct for each $t$ a continuously (in the sense of
Hausdorff metric) varying arc $L^t$ landing on $x^t$. This arc
 is contained in the image under the exponential map of the
 equidistant set  in Section~\ref{Sprelim}.

 \begin{lem}\label{metricarc} Let $O$ be an accessible point of $X$.
 Then there exists for each $t$ an arc $L^t$ such that $X^t\cap L^t=\{O^t\}$
  is an endpoint of $L^t$ and the function $\ba:[0,1]\to C(\complex)$ defined by $\ba(t)=L^t$ is a continuous function
  to the space $C(\complex)$ of compact subsets of $\complex$ with the Hausdorff metric.
  \end{lem}

\begin{proof}
We assume as usual that  $h^t(O)=O$ and $X^t\subset B(O,1)$  for all
$t$. Since every half ray in the plane is tame, we may assume that
the positive $x$-axis is contained in $\complex\setminus X^0$. Then
$\hX\subset \complex\setminus \pi_2^{-1}(\{0\})$. Let $\hA^0=\hX\cap
\pi_2^{-1}((0,\infty))$ and $\hB^0=\hX\cap \pi_2^{-1}((-\infty,0))$. Then $\hX$ is the
union of these two disjoint closed sets and $\hA^0$ lies above $\hB^0$. Since
$\hX$ is invariant under vertical translation by $2\pi$, it follows
that $E(\hA^0,\hB^0)$ is contained in $\pi_2^{-1}([-2\pi,2\pi])$. By
Lemma~\ref{order}, $\hA^t$ lies above $\hB^t$ for each $t$. By
Theorem~\ref{connectedE}, $E(\hA^t,\hB^t)$ is a ray which separates
$\hA^t$ and $\hB^t$ in $\complex$ and
$\pi_1(E(\hA^t,\hB^t))=(-\infty,\infty)$.

Let $t_i\to t_0\in[0,1]$. Then $\hA^{t_i}\to \hA^{t_0}$ and
$\hB^{t_i}\to \hB^{t_0}$ on compact sets (i.e., $K$ compact in $\hA^0$ implies
$K^{t_i}\to K^{t_0}$). It is easy to check that if $e_i\in
E(\hA^{t_i},\hB^{t_i})$ and $e_i\to e$, then $e\in E(\hA^{t_0},\hB^{t_0})$.

By Theorem~\ref{compactV}  $|E(\hA^t,\hB^t)\cap \pi_1^{-1}(1)|=1$.
For each $t$, let $M^t=E(\hA^t,\hB^t)\cap\pi_1^{-1}((-\infty,1])$. Then
$M^t$ is connected and $\lim M^{t_i}=M^{t_0}$. Hence $L^t=\ol{\exp(M^t)}$ is the required arc.
\end{proof}

By a \emph{crosscut} $C$ of a non-separating continuum $X\subset \complex$
we mean an open arc $C\subset\complex\setminus X$ whose closure is a closed arc with
 distinct endpoints $a$ and $b$  which are
in $X$. In this case we say that the crosscut $C$ \emph{joins} the points $a$ and $b$ of $X$. By
the \emph{shadow} of $C$, denoted by $Sh(C)$,
 we mean the closure of the bounded complementary domain of $\complex\setminus [X\cup C]$.

\begin{thm}\label{contarc}
Suppose that $O\in \bd(X)$, $h^t:\bd(X)\to\complex$ is an isotopy such that $h^0=id_{\bd(X)}$, $h^t(O)=O$
and $\diam (X^t)<1$ for all $t$. Let $U^t$ be the component of $\sphere\setminus h^t(\bd(X))$
containing $\infty$, let $\vp^t:\disk\to U^t$ be the normalized
Riemann map
 and let $L^t$ be an arc with one endpoint at $O$ such that $L^t$ varies continuously with $t$ in the Hausdorff metric
 and $L^t\cap X^t=\{O\}$ for all $t$. Then the function $\al:[0,1]\to S^1$ defined by
 $\al(t)=S^1\cap\ol{(\vp^t)^{-1}(L^t)}$ is continuous and the external ray $\vp^t(\{r e^{2\pi i \al(t)}
 \mid r<1\})=R^t_{\al(t)}$ lands on $O$ in $X^t$ for each $t$.
\end{thm}

We shall refer to the function $\al:[0,1]\to S^1$ in Theorem~\ref{contarc}
as the \emph{continuous angle function}.

\begin{proof} By \cite{miln00},  $\al$ as defined in the statement of the Lemma is a function.
It remains to be shown that $\al$ is
continuous. We will first present an outline of the proof.
Fix $\e>0$. Let $a(t_0)$ and $b(t_0)$
be endpoints of a crosscut $C(t_0)$ of $X^{t_0}$ in $B(O,1/2)$ such that $L^{t_0}$ lands in the shadow
of the crosscut $C(t_0)$ and $\diam((\vp^{t_0})^{-1}(C(t_0)))<2\e/3$. Choose $\ba<1/5\ \min(d(a(t_0),b(t_0)),
d(L^{t_0},\{a^{t_0},b^{t_0}\}), d(O, C(t_0)))$.
We shall choose $K$,  a large compact subarc of $C(t_0)$, such that $B(L^{t_0},\ba)\cap C(t_0)\subset K$
and
such that $L^t\subset B(L^{t_0},\ba)$ and  $K\cap X^t=\0$  whenever $t$ is close to $t_0$.
We shall define $K^t$, a crosscut of $X^t$, which contains a large sub-arc of $K$
 together with two small arcs $J(a,t)$ and $J(b,t)$
 which join points close to the endpoints of $K$ to $X^t$ such  that
$(\vp^t)^{-1}(K^t)$  is a small crosscut of $\disk$ whose shadow  contains $\al(t_0)$
and $\al(t)$ for $t$ sufficiently close to $t_0$.

Let  $C(t_0)$, $a(t_0)$, $b(t_0)$ and $\beta$ be defined as above
and let $\hat{a}(0)$ and $\hat{b}(0)$ be the endpoints of $(\vp^{t_0})^{-1}(C(t_0))$. We may assume that
$\al(t_0)$ is contained in the interval $(\hat{a}(0),\hat{b}(0))$ in the boundary circle $S^1$
contained in the shadow of $(\vp^{t_0})^{-1}(C(t_0))$. Then
$|\hat{b}(0)-\hat{a}(0)|<2\e/3$.
Choose $\da_1>0$ such that
$\da_1<(1/5) \min(|\hat{b}(0)-\al(t_0)|,|\al(t_0)-\hat{a}(0)|,\e/4,\ba)$.
Let $\rho<\sqrt{\rho}<\min\{\da_1,1\}$ such that $\frac{2\pi}{\sqrt{\ln(1/\rho)}}<\da_1$ and
\begin{gather}
(\vp^{t_0})^{-1}(B(a(t_0),2\sqrt{\rho})\cap C(t_0))\subset B(\hat{a}^{t_0},\da_1),\label{1r}\\
(\vp^{t_0})^{-1}(B(b(t_0),2\sqrt{\rho})\cap C(t_0))\subset B(\hat{b}^{t_0},\da_1)\label{2r},
\end{gather}
and there is just one   component $K$ of $C(t_0)\setminus [B(a(t_0),\rho/2)\cup B(b(t_0),\rho/2)]$
which meets  both $S(a(t_0),\rho/2)$ and $S(b(t_0),\rho/2)$.

Next choose $\da_2>0$ such that for all $|t-t_0|<\da_2$:
\begin{enumerate}
\item $L^t\subset B(L^{t_0},\rho/2)$,
\item $d(h^t,h^{t_0})<\rho/2$,
\item $X^t\cap K=\0$ and
\item \label{4r}$d((\vp^t)^{-1}|_K, (\vp^{t_0})^{-1}|_K)<\da_1$.
\end{enumerate}

By \cite[Prop. 2.2]{pomm92} there exist $\rho\le r,s\le \sqrt{\rho}$ such that if
$J(a,t)$ is a component of $S(a(t_0),r)\setminus X^t$ which meets $K$ and
and $J(b,t)$ is a component of $S(b(t_0),s)\setminus X^t$ which meets $K$, then
\begin{equation}\label{3r}
\diam(\vp^t)^{-1}(J(z,t))\le \frac{2\pi}{\sqrt{\ln(1/\rho)}}<\da_1
\text{ for } z\in\{a,b\}. \end{equation}

 Then $K\cup J(a,t)\cup J(b,t)$
contains a crosscut $C(t)$ of $X^{t}$ and  $L^{t}$ lands in the shadow of $C(t)$.
Since the endpoints of $C(t)$ are joined by a subarc of $J(z,t)\subset B(z,\sqrt{\rho})$ to $K$
it follows from (\ref{1r}), (\ref{2r}), (\ref{3r}) and (\ref{4r}) that
 the endpoints of $(\vp^t)^{-1}(C(t))$ are within $3\da_1<(3/20)\e$ of the endpoints
 $\hat{a}^{t_0}$ and $\hat{b}^{t_0}$ of $(\vp^{t_0})^{-1}(C(t_0))$.
 Then for $|t-t_0|<\da_2$, $\al(t_0)$ is in the shadow of $(\vp^{t})^{-1}(C(t))$,  $\vp^t(L^t)$ lands in this shadow
 and the distance between the endpoints of $(\vp^{t})^{-1}(C(t))$ is less than $6\da_1+2\e/3<
 (6/20+2/3)\e<\e$ as desired.
\end{proof}

\begin{thm} Suppose $h^t$ is an isotopy of the boundary of a non-separating continuum $X\subset \complex$
 such that $h^0=id_{\bd(X)}$. Let $U^t$ be the  component of $\sphere\setminus h^t(\bd(X))$ containing $\infty$
 and let $\vp^t:\disk\to U^t$ denote the normalized Riemann map. Then there exists
an isotopy $\al^t:S^1\to S^1$ such that $\al^0=id_{S^1}$ and if $R^0_\ta$ lands on $x^0\in X^0$, then
$R^t_{\al^t(\ta)}$ lands on $x^t$ for each $t$.
\end{thm}

\begin{proof}Suppose that $R^0_\ta$ lands on $x^0\in X^0$. By Theorem~\ref{contarc}, there exists a continuous function
$\al_\ta:[0,1]\to S^1$ such that $\al_\ta(0)=\ta$ and $R^t_{\al_\ta(t)}$ lands on $x^t$ for each $t$.
Let $\mathcal{A}$ be the set of angles in $S^1$ such that for each $\ta\in\mathcal{A}$, $R^0_\ta$
lands on a point $x(\ta)\in X^0$. Define $\al^t:\mathcal{A}\to S^1$ by
$\al^t(\ta)=\al_{\ta}(t)$. Then $\al^t$ is a circular order preserving isotopy
of $\mathcal{A}$ such that $\al^0=id_{\mathcal{A}}$.
 Since $\mathcal{A}$ is dense in $S^1$, $\al^t$ can be extended to an isotopy of $S^1$.
\end{proof}

We will refer to the isotopy $\al^t$ as the \emph{continuous angle isotopy}.

\section{Extension over hyperbolic crosscuts} Suppose $U$ is a component of $\complex\setminus Z$
and $h^t:Z\to\complex$ is an isotopy
 such that $h^0=id_{Z}$. Then there exists a path $\ga:[0,1]\to\complex$ such that
 $\ga(0)\in U$ and $\ga(t)\in\complex\setminus Z^t$ for all $t$.
 Then we denote by $U^t$ the component of $\complex\setminus Z^t$ which contains the point
 $\ga(t)$.  Note that $U^t$ is independent of the choice of the path $\ga$.
 We have shown in
the previous section that if there exists a  crosscut
in $U^0$ joining the points $a^0$ and $b^0$, then for each $t$ there exists a
crosscut in $U^t$ joining the points
$a^t$ and $b^t$. We will show next that we can choose for each $t$ a natural crosscut $C^t$
joining these points such that
the isotopy $h$ can be extended over $X^0\cup C^0$. For this purpose we will use
hyperbolic geodesics defined by  the Poincar\'e metric on $\disk$.

Suppose that $a^0$ and $b^0$ are the landing points of the external rays $R^0_{\ta(a)}$ and $R^0_{\ta(b)}$
in $\complex\setminus X^0$. By Theorem~\ref{contarc}, there exist continuous angle functions $\al:[0,1]\to S^1$ and
$\ba:[0,1]\to S^1$ such that for each $t$, $R^t_{\al(t)}$ and $R^t_{\ba(t)}$ land on $a^t$ and $b^t$, respectively.
Let $G^t$ be the hyperbolic geodesic joining the points $\al(t)$ and $\ba(t)$ in $\disk$ (i.e., $G^t$
is the intersection of the round circle through the points $\al(t)$ and $\ba(t)$ with $\disk$
which crosses $S^1$ perpendicularly at
both of these points). Let $C^t=\vp^t(G^t)$. We will call $C^t$ the \emph{hyperbolic crosscut of $X^t$ joining the
points $a^t$ and $b^t$}.
In the final part of this section we will consider  $Z$ as a subset of the sphere and show
 that the isotopy $h:X^0\times [0,1]\to\sphere$ can be extended to an isotopy $H:X^0\cup C^0\times [0,1]\to\sphere$
such that $H^t(C^0)=C^t$, where
$C^t$ is the hyperbolic crosscut of $X^t$ joining $a^t$ to $b^t$.
 We will make
use of the following  well-known Theorem
\cite[Theorem 4.20]{pomm92}\footnote{We are indebted to Paul Fabel for this reference.}.
\begin{thm}[Gehring-Hayman Theorem]\label{GHthm}
There exists a universal constant $K$ such that
for for any conformal map $\vp:\disk\to\complex$,
if $z_1,z_2\in\ol{\disk}$, $\ga$ is an arc in $\disk$ from $z_1$ to $z_2$,
and $S$ is the hyperbolic geodesic from $z_1$ to $z_2$, then $ \diam(\vp(S))\le K\, \diam(\vp(\ga))$.
\end{thm}

Recall that $X$ is a non-separating plane continuum.
Each  angle $\ta\in S^1$ corresponds to a \emph{prime end} of $\complex^*\setminus X$. By a fundamental chain
$C_j$ of crosscuts we mean a sequence of crosscuts of $X$ such that $\lim\diam(C_j)=0$, $C_i\subset
Sh(C_j)$
for $i>j$ and the arcs $\{\ol{C_i}\}$ are all pairwise disjoint. A naturally defined
equivalence class of fundamental chains is called a prime end of $\complex^*\setminus X$
(see \cite{miln00} for further details).

\begin{lem}\label{extendray}Let $h$ be an isotopy of $\bd(X)$, $O\in \bd(X)$ and $h^t(O)=O$ for all $t$.
Suppose that $R^0_\ta$ is a conformal external ray of $X^0$ landing on $O$.
 Then the isotopy
$h$ can be extended to an isotopy $H:[X\cup R^0_\ta]\times[0,1]\to\complex$ such that $H^t(R^0_\ta)$ is an external ray
of $X^t$ landing on $O$.
\end{lem}

\begin{proof}
 By Lemma~\ref{contarc}, there exists a continuous angle function $\al:[0,1]\to S^1$
such that $\al(0)=\ta$ and the (conformal) external ray $R^t_{\al(t)}$ lands on $O$ for each $t$.
Extend the isotopy $h$ over $R^0_{\al(0)}$ by
\begin{equation}\label{defH} H(z,t)=\vp^t\circ \rho^t \circ (\vp^0)^{-1}(z)\end{equation}
 for $z\in R^0_{\al(0)}$,
where $\rho^t$ is the rotation of $\disk$ by the angle $\al(t)-\al(0)$. By Carath\'eodory kernel convergence,
$H$ is an isotopy of every compact subset of $R^0_{\al(0)}$. Hence it suffices to show that if $z_i\to O$ in
$R^0_{\al(0)}$ and $t_i\to t_{\infty}$,
then $H(z_i,t_i)\to O=H(O,t_\infty)$.

To see this fix $\e>0$.
It suffices to  show that there exists an open  disk $B$ containing $O$ with simple closed
curve boundary $S$ and $\da>0$ such that
for all $t$ with $|t-t_\infty|<\da$, if
$z^t$ is the first point of $R^t_{\al(t)}$ (from $\infty$) on $S$ and if $CR^t_{z^t}$ is the
component of $R^t_{\al(t)}\setminus z^t$ from $z^t$ to $O$, then $CR^t_{z^t}\subset B(O,\e)$.

Let $K$ be the universal constant from
Theorem~\ref{GHthm}. By Lemma~\ref{metricarc} there exists a continuously varying arc $L^t\subset\complex\setminus X^t$
landing on $O$ in $X^t$ for each $t$ such that $\ol{(\vp^0)^{-1}(L^0)}\cap S^1=\{\ta\}$. Choose a
fundamental chain of crosscuts $C^{t_\infty}_n$ of $X^{t_\infty}$
for the prime-end $\al(t_\infty)$. Then both $L^{t_\infty}$ and $R^{t_\infty}_{\al(t_\infty)}$
cross $C^{t_\infty}_n$ essentially
(that is $X\cup C^{t_\infty}_n$ separates the endpoints of $L^{t\infty}$ and also the ends of
the ray $R^{t_\infty}_{\al(t_\infty)}$).
Hence we can choose $n$ sufficiently large and a simple closed curve $S$ containing $O$
in its bounded complementary domain $B$ such that
 $C^{t_\infty}_n\subset S$, $[L^{t_\infty}\cup R^{t_\infty}_{\al(t_\infty)}]
\cap [S\setminus C^{t_\infty}_n]=\0$ and $\diam(S)<\e/K$. From now on fix this $n$
and let $a$ and $b$ be the endpoints of $C^{t_\infty}_n$.

For $t$ close to $t_\infty$, let $w^{t}$ be the first point (from $O$) of $L^t$ on $S$.
Let $C^t$ be the component of
$S\setminus X^t$ containing the point $w^{t_\infty}$.
Choose $\rho<(1/3) d(\{a,b\},[L^{t_\infty}\cup  R^{t_\infty}_{\al(t_\infty)}])$  and
let $C^t_-$ be the component of
$C^t\setminus [B(a,\rho)\cup B(b,\rho)]$ which contains $w^{t_\infty}$. Choose
$\da>0$ such that if $|t-t_\infty|<\da$, then
\begin{enumerate}
\item $w^{t_\infty}\in \complex\setminus [X^t\cup B(a,\rho)\cup B(b,\rho)]$,
\item \label{disjoint} $C^t_-=C^{t_\infty}_-$,
\item \label{inU} $L^t\subset B(L^{t_\infty},\rho)$,
\item \label{firstz} if $z^t$ is the first point of $R^t_{\al(t)}$ (from infinity) on $S$, then $z^t\in C^t_-$.
\end{enumerate}
The first and second conditions follow from  the continuity of $h$
 and the third from the continuity of $L^t$. The last condition follows from Carath\'eodory kernel convergence:
recall that $d(R^{t_\infty}_{\al(t_\infty)}, S\setminus C^{t_\infty})=\eta>0$. Let
$v\in R^{t_\infty}_{\al(t_\infty)}\cap B$ such that the component of $R^t_{\al^t(\ta)}\setminus \{v\}$
from $v$ to $O$ is contained in $B$ and let $(\vp^{t_\infty})^{-1}(v)=r_0\exp(\al(t_\infty))$.
By Carath\'eodory kernel convergence,
$I^t=\vp^t(\{r\exp(\al(t))\mid 0\le r \le r_0\})$ converges to the  segment from $v$ to $\infty$
in $R^{t_\infty}_{\al(t_\infty)}$.
Hence $d(I^t, S\setminus C^{t_\infty})>\eta/2$ and $d(I^t,v)<(1/2) d(v,S)$ for $t$ close to $t_\infty$, and
(\ref{firstz}) hold for $\da$ sufficiently small.

 By (\ref{disjoint}),
(\ref{inU}) and (\ref{firstz}),
the sub-arc $A^t$ of $C^t$ joining the points $w^t$ and $z^t$, is contained in $\complex\setminus X^t$.
Hence the union of the arcs $A^t$ and $[w^t,O]\subset L^t$ is an arc in $[\complex\setminus X^t]\cup \{O\}$, joining
$z^t$ to $O$, of diameter less than $\e/K$. By Theorem~\ref{GHthm}, the terminal segment
 $CR^t_{z^t}\subset B(O,\e)$ as required.

\end{proof}

In  the remaining part of the paper we will consider $Z$ as a subset of the unit sphere $\sphere\subset \real^3$
 with \emph{spherical metric} $\rho$. Hence
the distance between two points $z,w\in\sphere$ is the length of the shortest arc in the great circle which
is the intersection of $\sphere$
and the plane through $z$, $w$ and the origin in $\real^3$.

Since every hyperbolic crosscut is conformally equivalent to a diameter of $\disk$ it follows that
we can extend the isotopy $h^t$ over any hyperbolic crosscut $C^0\subset U^0$ joining two
points $a^0$ and $b^0$ in $Z^0$ to an isotopy $H:Z^0\cup C^0\to\sphere$
(since in this case the point at infinity is not fixed, the range of the isotopy must be the sphere).
 Note that if $C_i$ is a convergent sequence of hyperbolic
crosscuts whose limit contains a non-degenerate subcontinuum $Y\subset Z$, then this extension
of the isotopy over $\cup C_i$ is not necessarily continuous at $Y$. However,
 we can  extend over a suitable compact set of hyperbolic crosscuts in $U$ as follows.
At this point it will be convenient to change to the Cayley-Klein model\footnote{We are indebted to Nandor Simanyi
 for suggesting the  Cayley-Klein model.} of the hyperbolic disk. There exists a
homeomorphism $g:\ol{\disk}\to\ol{\disk}$, which is the identity on the boundary $S^1$ of $\disk$,
such that $g$ preserves radial line segments and for any two points $\ta_1,\ta_2\in S^1$ the hyperbolic geodesic
$G$ joining $\ta_1$ to $\ta_2$ is mapped to the straight line segment  $\ta_1\ta_2$, which is a chord of the unit
disk with endpoints $\ta_1$ and $\ta_2$.

Suppose that $\Hy$ is a collection of disjoint hyperbolic crosscuts in $U$ such that
the set $\bigcup_{c\in\Hy} \ol{C}$ is compact, then we call $\Hy$ a
\emph{compact set of disjoint hyperbolic crosscuts in $U$}.
The compactness implies that there exists $\e>0$ such that for each
$C\in\Hy$, $\diam(C)\ge \e$. Let $a_C$ and $b_C$ denote the endpoints of $C\in\Hy$, let
$\al_C$ and $\ba_C$ be the corresponding endpoints of $(\vp^0)^{-1}(C)$ and let
$\mathcal{A}^0$ denote the union of all the angles $\{\al_C,\ba_C\}$ for $C\in\Hy$. Let $\al^t$ be the continuous angle
isotopy and let $\mathcal{A}^t=\al^t(\mathcal{A}^0)$. Then for each $t\in[0,1]$, the collection of chords
$\al^t(\al_C)\al^t(\ba_C)$ is a compact lamination in the unit disk in the sense of Thurston \cite{thur85}.
We will denote the  family of all such chords by $\lam^t$. We will say that $\lam^0$ is the \emph{pullback
of the lamination $\Hy$ to the unit disk.} Let  $\mathcal{L}^{t^*}=
\bigcup \mathcal{L}^t$.
Note that any two distinct  chords in $\lam^t$  meet at most in a common endpoint and there exists $\da>0$ such that
for each $t$ and each chord in $\lam^t$, $\diam(\al^t(\al_C)\al^t(\ba_C))>\da$. Let $L^t:\lam^0\to \lam^t$ be the
\emph{linear isotopy on  $\lam$} which extends $\al^t$ and maps each  chord
$\al^0(\al_C)\al^0(\ba_C)$ linearly onto the chord $\al^t(\al_C)\al^t(\ba_C)$.  Then  the following theorem follows.

\begin{thm}\label{contcross} Suppose that $\Hy=\Hy^0$ is a compact set of disjoint hyperbolic crosscuts in $U^0$.
 Then the isotopy $h:Z^0\times [0,1]\to\sphere$
can be extended to an isotopy $H:[Z\cup \Hy^{*}]\times[0,1]\to \sphere$ such that $H^t(\Hy^*)=\Hy^{t^*}
=\vp^t(g^{-1}(\lam^{t^*}))$ and $H$ is defined by:
\[H^t(z)=
\begin{cases} h^t(z),   &\text{if $z \in Z^0$;} \\
\vp^t\circ g^{-1}\circ L^t \circ g\circ (\vp^0)^{-1}(z)  &\text{if $z\in \Hy^*$,}
\end{cases}\]
where $L^t$ is the linear isotopy on the pullback $\lam$ of $\Hy$
\end{thm}
 We will say that the extended isotopy $H$ defined in Theorem~\ref{contcross} is the
 \emph{linear extended isotopy which preserves the hyperbolic crosscuts} in $\Hy$.

\section{Existence of short crosscuts}
It follows from the results of the previous section that if $C$ is the hyperbolic crosscut of $Z$
 which joins the points $a$ and $b$ in a complementary domain $U$ of $Z$ in $\sphere$,
 then we can extend the isotopy to an isotopy $H$ of $Z\cup C$ such
 that $H^t(C)=C^t$ is the hyperbolic crosscut joining the
 points $a^t$ and $b^t$.  We need to show that if the crosscut $C$ has small diameter, then  the
 crosscut $C^t$ also has small diameter.
 If $C$ is contained in the component $U$ of $\sphere\setminus Z$, then we denote
 by $U^t$ the component of $\sphere\setminus Z^t$ which contains $C^t$.

Given a hyperbolic crosscut $C$ of a  continuum $Z\subset \sphere$,
 we say that $C$ is a \emph{$\da$-hyperbolic crosscut} if the diameter
 of $C$ is less than $\da$. Note that we see $Z$ as a subset of the sphere $\sphere$ with the
 spherical metric $\rho$.

\begin{thm}\label{smallC} For each $\e>0$ there exists $\da>0$ such that if $x,y\in Z$ can be joined by a
$\da$-hyperbolic crosscut $C\subset U$, where $U$ is a component of $\sphere\setminus Z$, then for each $t$, $x^t$
and $y^t$ are joined by an $\e$-hyperbolic crosscut in $U^t$.
\end{thm}

\begin{proof}
 Suppose that the Theorem fails for some $\e>0$. Then there exist $x_n,y_n\in Z^0=Z$ and
 a sequence of $1/n$-hyperbolic crosscuts
 $C_n$ in complementary domains $U_n$ joining them and $t_n\in[0,1]$
 such that the points $x^{t_n}_n$ and $y^{t_n}_n$
  are not  joined by an $\e$-hyperbolic crosscut in $U^{t_n}_n$. Then there exists $0<\e'<\e$ and
  $w_n\in \bd(U_n)$, accessible from $U_n$, such that $\rho(w_n^t,O)>\e'$ for all
   $n$ and all $t\in[0,1]$.
  Without loss of generality, the origin $O\in Z$,
   $\lim C_n=\{O\}$ and $h^t(O)=O$ for all $t$.

Let $K$ be the universal constant from Theorem~\ref{GHthm}.
Choose $0<\da<\e'/3$ such that if $h^t(z)\in \ol{B(O,\da)}$ for some $t\in[0,1]$, then $h^s(z)\in B(O,\e/3K)$ for all
$s\in [0,1]$. Choose $n_0$ such that $C_{n_0}\subset B(O,\da)$ and $\{x^s_{n_0},y^s_{n_0}\}\subset B(O,\da)$ for
all $s\in[0,1]$. From now on we fix this $n=n_0$ and, hence, we can omit $n$ from the notation. In particular
we have a fixed component $U$ of $\sphere\setminus Z$, three points $x,y,w\in\bd(U)$ with
$x$ and $y$ joined by the hyperbolic
crosscut $C\subset U\cap B(O,\da)$, with $x^s,y^s\in B(O,\da)$ and $w^s\in\complex\setminus \ol{B(O,\da)}$
 for all $s\in[0,1]$.  By Theorem~\ref{contcross}
we can extend the isotopy $h$ to an isotopy $H$ of $Z\cup C$
 such that $H^t(C)=C^t$ is the hyperbolic crosscut joining $x^t$ to
$y^t$ in $U^t\subset \sphere\setminus Z^t$ for each $t$.

Let $D$ be the closed $\da$-ball centered at $O$. For each $t\in[0,1]$,
let $P^t=D\cup C^t$. Since $\bd(P^t)$ is contained in $S(O,\da)\cup C^t$,
which is a finite union of arcs, $\bd(P^t)$ contains no continuum
of convergence and each sub-continuum of $\bd(P^t)$ is locally connected
and arcwise connected \cite{whyb42}.

Since $C^t$ is an arc, the components  $\{T_i\}$ of $C^t\setminus D$ form a
null sequence. For each $i$, $\ol{T_i}$ is an arc and $\ol{T_i}\cap D$
consists of the endpoints of $T_i$. Each point of $\sphere\setminus P^t$ can be joined
to $w^t$ by an arc in $\sphere\setminus D$ which meets $P^t$ in a finite set.

Suppose that $V$ is a component of $\sphere\setminus P^t$. We say that $V$ is an \emph{odd domain}
(respectively \emph{even domain}) \emph{of $\sphere\setminus P^t$} if there is a closed arc
$A\subset \sphere\setminus D$ from $w^t$ to
a point in $V$ such that $|A\cap C^t|$ is odd (respectively, even) and $A$ is transverse to $C^t$ at each point of
$A\cap C^t$. This definition is independent of the choice of the arc and the point in $V$.

Let $Q^t=P^t\cup \bigcup\{V\mid V \text{ is an odd domain of } \sphere\setminus  P^t\}$. The boundary
of each odd domain $V$ of $\sphere\setminus P^t$ is a simple closed curve which meets $D$ and there exists
a $T_i$ such that $T_i$ is  contained in $\bd(V)$ and $T_i\cup D$ separates $V$ from $w^t$ in $\complex\setminus D$.
Also each $T_i$ is contained in the boundary of exactly one odd domain of $\sphere\setminus  P^t$.
Since the odd domains form a null family, $Q^t$ is a locally connected continuum.

Let $t_i$ converge to $t\in[0,1]$. We prove that $\lim Q^{t_i}=Q^t$. Note that $\lim P^{t_i}=P^t$.
Let $z\in \sphere\setminus P^t$. It suffices to prove that $z\in Q^t$ if and only if
$z\in Q^{t_i}$ for all sufficiently large $i$. Let $A\subset \sphere\setminus B(O,\rho(O,z))$
 be a piecewise linear arc from $z$
to $w^t$ which witnesses whether or not $z\in Q^t$. Then $A$ meets only
finitely many, without loss of generality $T_1,\dots,T_n$, of the open arcs $T_j$.
Let $H:Z\cup C\to\sphere$ be the extended linear isotopy of Theorem~\ref{contcross} such
that $H^t(C)=C^t$. Let $\da<\da'<\min(\rho(z,O),2\e'/3)$ then for all
$i$ sufficiently large $\ol{B(O,\da')}\cup T_{j}$ separates $z$ from $w^t$
if and only if $\ol{B(O,\da')}\cup H^{t_i}((H^t)^{-1}(T_{j}))$ does for each $j=1,\dots,n$ and
$T_j\cap A\ne\0$ if and only if $H^{t_i}((H^t)^{-1}(T_j))\cap A\ne\0$ for all $j$.

Note that $|T_j\cap A|$ is odd if and only if $\ol{B(O,\da')}\cup T_j$ separates
$z$ from $w^t$. Fix any large $i$ and choose an arc $M$ very close to $A$ which witnesses
 whether $z$ is in $Q^{t_i}$. Then $|M\cap H^{t_i}((H^t)^{-1}(T_j))|=|A\cap T_j|\negthickspace\mod 2$
 for $j=1,\dots,n$ and $M\cap H^{t_i}((H^t)^{-1}(T_j))=\0$ for all $j>n$. Hence
 $z\in Q^{t_i}$ if and only if $z\in Q^t$ as desired.

 Let $z^t\in Q^t\cap Z^t$. We prove that $\rho(z^t,O)<\e/3K$. We may assume that $z^t\not\in\{x^t,y^t\}\cup D$.
 Let $s_0=\inf\{s\in[0,1]\mid z^s\in Q^s\setminus D\}$. Since $Q^0=D$ and $Z^s\cap C^s=\0$ for all $s$,
 $z^{s_0}\in D$. Hence, by the choice of $\da$, $\rho(O,z^t)<\e/3K$.

It remains to prove the following:

 \emph{Claim.} $Q^t\cap B(O,\e/2K)$ contains an arc $A$ such that $A\cap Z^t=\{x^t,y^t\}$.

 Proof of Claim. Fix $t\in [0,1]$. Then $C^t$ is an arc such that $C^t\cap Z^t=\{x^t,y^t\}$.
 After a small perturbation of $C^t$ we may assume that $C^t\cap S(O,\e/2K)$ is finite and
 all intersections are transverse. Note that the definition of an odd domain
 of $\sphere\setminus P^t$
 was with respect to $P^t=D\cup C^t$.
 In what follows we will use the same definition but now with respect
 to  $P^t_M=D\cup M$, where $M\subset U^t\cup\{x^t,y^t\}$ is an arc such that:
 \begin{enumerate}
 \item $M\cap Z^t=\{x^t,y^t\}$ and $x^t$ and $y^t$ are endpoints of $M$,
  \item $M\cap S(O, \e/2K)$ is finite and all intersections are transverse,
 \item \label{smallZ}
 for each odd domain $V$ of $\sphere\setminus P^t_M$ and each $z^t\in Z^t\cap V$, $\rho(z^t,O)<\e/3K$,
 \item  $n=|M\cap S(O,\e/2K)|$ is minimal.
 \end{enumerate}

 If $n=0$ we are done. Note that $n=1$ is impossible since all
 intersections of $S(O,\e/2K)$ and $M$ are transverse and both
 endpoints of $M$ are in $B(O,\e/2K)$. Hence, assume $n>1$.    Let $Q^t_M=P^t_M\cup \bigcup\{V \mid V
 $ is an odd domain of  $\sphere\setminus  P^t_M\}$.
  Let $S_i$ be all components of $S(O,\e/2K)\setminus M$. Since
 each component $M_j$ of $M\setminus D$ is an arc
 which locally separates the plane, points on one side of $M_j$ are
 in  an even domain and points on the other side are in an odd
 domain. Hence, each arc $S_i$ is contained in a complementary
 domain $V_i$ of $\sphere\setminus  P^t_M$ and these domains are
 alternately even and odd moving around the circle $S(O,\e/2K)$. In
 particular, $n$ is even. We may order $M$ so
 that $x^t<y^t$ and we  write intervals
 in $M$ as in $\real$.

Let $\M=\{M_i\}$ be the collection of all components of $M\setminus
\ol{B(O,\e/2K)}$. We can define a partial order $\prec$ on $\M$ by
$M_1\prec M_2$ if $M_2$ separates $M_1$ from $w^t$ in
$\complex\setminus \ol{B(O,\e/2K)}$. Assume that $M_1=(a_1,b_1)$
(with $a_1<b_1$)  is a minimal element of $\M$. Then $M_1\cup
\ol{B(O,\e/2K)}$ bounds a disk $V_1$ whose closure meets
$\ol{B(O,\e/2K)}$ in an arc $S_1\subset S(O,\e/2K)$ and $S_1\cap
M=\{a_1,b_1\}$. Then $S_1$ is either contained in an an even or an odd
domain of $\sphere \setminus P^t_M$.

Subcase 0. Suppose that  $Z^t\cap S_1=\0$ (this must be the case if $S_1$ is contained
in an odd domain). In this case choose
$a'_1<a_1<b_1<b'_1$, with $a'_1$ in  $B(O,\e/2K)$ very close to $a_1$ and $b'_1$ in $ B(O,\e/2K)$ very
close to $b_1$, and an arc $S'_1\subset B(O,\e/2K)$ very close to
$S_1$ from $a'_1$ to $b'_1$ such that $S'_1\cap Z^t=\0$. Then $Z^t$ is disjoint from the
bounded complementary domain $B$ of the simple closed curve $F=S'_1 \cup (a'_1,b'_1)$. Hence
there exists a homotopy of the plane which is the identity on
$Z^t\cup S'_1\cup [x^t,a'_1]\cup [b'_1,y^t]$ and shrinks $B$ to
$S'_1$. Let $M'=S'_1\cup [M\setminus (a_1,b_1)]$. Then $z\in Z^t$ lies in an odd domain of
$\sphere\setminus P^t_M$ if and only if z lies in
an odd domain of $\sphere\setminus[D\cup M']$. Thus
$M'$ satisfies (\ref{smallZ}). Clearly $M'$ satisfies (1-2) and
since $|M'\cap S(O,\e/2K|<n$ we have a contradiction with the
minimality of $n$.

Hence we may assume that  $S_1\cup V_1$ is contained in an even domain and
$Z^t\cap S_1\not=\0$. Then there
exists $M_2=(a_2,b_2)\in \M$ (with $a_2<b_2$) such that $M_2$  is the immediate successor of $M_1$
in $\M$. Let $V_2$ be the component of $\sphere\setminus [\ol{V_1\cup  B(O,\e/2K)}\cup M_2]$ whose
closure contains the arc $(a_1,b_1)$. Since $M_2$ is the immediate successor of $M_1$
in $\M$, there exists an arc $J\subset [V_2\setminus M]\cup\{j_1,j_2\}$ with one endpoint of $J$,
$j_1\in(a_1,b_1)$ and the other endpoint of $J$, $j_2\in (a_2,b_2)$. Moreover,
since $V_1$ was even, $J$ is contained in an odd domain and $J\cap Z^t=\0$.
We will examine the circular (counter clockwise) order $<_C$ of the four points
$a_1,b_1,a_2,b_2$ around the circle $S(O,\e/2K)$.

\begin{figure}\refstepcounter{figure}\label{subcase1}\addtocounter{figure}{-1}
\begin{center}\hspace*{-2.5in}
\includegraphics[width=6.5truein]{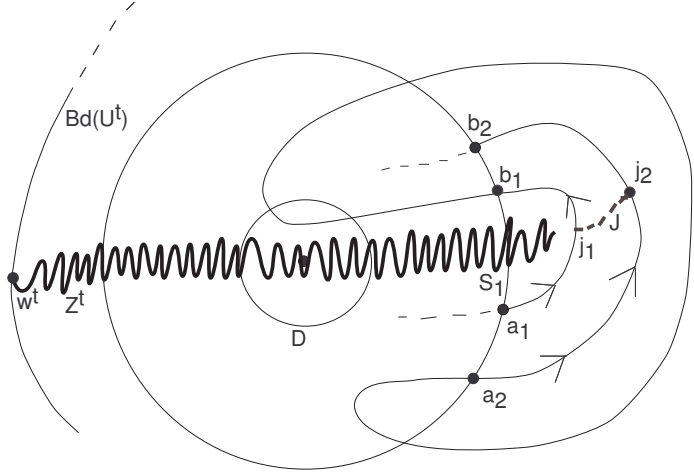}
\caption{Subcase 1  in the proof of Theorem~\ref{smallC}}
\end{center}
\end{figure}

Subcase 1. $a_2<_C a_1<_C  b_1 <_C b_2$. We have either $a_1<a_2<b_1<b_2$ (see figure~1)
or $a_2<b_2<a_1<b_1$. Since $w^t\in Z^t\cap \bd(U^t)$, $Z^t\cap S_1\not=\0$ and $M\cap Z^t=\{x^t,y^t\}$,
in either case (see the gate theorems in \cite[page 36]{beck74}),
the simple closed curve $F=J\cup [j_1,j_2]$ separates $x^t$ from $y^t$.
Since $Z^t\cap F=\0$, this contradicts the connectedness of $Z^t$.

Subcase 2. $b_2<_C a_1 <_C b_1<_C a_2$. Then either $a_1<b_1<a_2<b_2$ or $a_2<b_2<a_1<b_1$.
Since $M\cap Z^t=\{x^t,y^t\}$, $w^t\in Z^t\cap \bd(U^t)$ and $Z^t\cap S_1\not=\0$,
$x^t$ and $y^t$ are contained in the unbounded
component of $F=J\cup [j_1,j_2]$ and the proof proceeds as in Subcase 0, where $F$ is now
$J\cup [j_1,j_2]$.

Other cases are similar. This completes the proof of the Claim.

  Hence for each $t$ there exists a
crosscut $M(t)$ joining $x^t$ to $y^t$ in
$U^t$ such that $\diam(M(t))<\e/K$. By
Theorem~\ref{GHthm}, the diameter of the
hyperbolic crosscut $C^t$ is less than
$\e$ for all $t$. This contradiction
completes the proof.

 \end{proof}

\section{Extending the isotopy over $\complex$}

Now that we know how to extend the isotopy over hyperbolic crosscuts,
it remains to define the extension over all complementary domains $U$ of $Z$. Easy examples show that if we choose the
hyperbolic crosscuts without care the extension  may not be continuous.
Fortunately a suitable set of hyperbolic crosscuts exists. Fix a component $U$ of
$\sphere\setminus Z$ and let $\mathcal{B}$ be the collection of all maximal open balls $B(z,r)\subset U$
(that is open balls in the spherical metric and such that $|S(z,r)\cap Z|\ge 2$).
Let $\mathcal{C}$ be the collection of all centers of such balls
and for $c\in\mathcal{C}$ let $r(c)$ be the corresponding radius.  Note that for each $c\in\mathcal{C}$, $B(c,r(c))$ is
conformally equivalent with the unit disk $\disk$ and, hence, can be endowed with the
hyperbolic metric. Let
$F(c)$ be the convex hull of the set $S(c,r(c))\cap X$ in $B(c,r(c))$ using the
hyperbolic metric on the ball $B(c,r(c))$. The following Theorem is due to Kulkarni and Pinkall:

\begin{thm}[\cite{kulkpink94}]\label{KP} For each $z\in U$ there exists a unique $c\in\mathcal{C}$
such that $z\in F(c)$.
\end{thm}

Note that the collection of chords in the boundaries of all $F(c)$ form
a lamination  of $U$ in the sense of Thurston \cite{thur85}. As in \cite{thur85}
we will call the chords in this lamination leaves.
Then two such leaves do not cross each other (i.e., if $\ell\ne \ell'$ are leaves, then
$\ell\cap \ell'\cap U=\0$) and any convergent sequence of leaves is either a leaf,
 or a point in $Z$. In particular, the subcollection  of leaves of diameter
greater or equal to $\e$ is compact for each $\e>0$. This collection of leaves    will naturally
provide us with the required collection of hyperbolic crosscuts by simply replacing each leaf in
the lamination
 by the hyperbolic crosscut joining its endpoints. The  collection $\mathcal{H}$ of such hyperbolic crosscuts
will be called the \emph{hyperbolic KP-lamination of  $U^0$}.
The union of all the hyperbolic crosscuts in $\mathcal{H}$ will be denoted
by $\mathcal{H}^*$. A \emph{gap} $G$ of $\Hy$ is the closure of a component of $U\setminus \Hy^*$.
By Theorems~\ref{contcross} and \ref{smallC} we can extend the isotopy $h$ over $\Hy^*$.
To finish the proof we must extend the isotopy over all gaps.

\begin{thm}\label{mainsphere} Suppose that $h^t$ is an isotopy of a plane continuum $Z$,
which we consider as a subset of the sphere $\sphere$,
with $h^0=id|_Z$.
Then there exists an extension to an isotopy $H^t:\sphere\to\sphere$ such that $H^0=id_{\sphere}$.
\end{thm}
\begin{proof} Let $\{U_n\}$ be all the components of $\sphere\setminus Z$. For each $n$ let $\mathcal{H}_n$
be the hyperbolic KP-lamination of $U_n$. Since the diameter of maximal balls contained in distinct
components $U_n$ converges to $0$, it follows from Theorems~\ref{smallC} and \ref{GHthm} that
for any sequence $C_n\in\mathcal{H}_n$, such that $U_n\ne U_m$ when $n\ne m$,
\begin{equation}\label{smallcenter} \lim\diam(C_n)=0. \end{equation}

By Theorems~\ref{smallC} and \ref{contcross} we can extend the isotopy $h$ of $Z$ to
an isotopy $H_n$ of $Z\cup \mathcal{H}^*_n$ such that $H_n$ preserves the hyperbolic crosscuts
in $\mathcal{H}^0_n$ (i.e., $H^t_n(\mathcal{H}^0_n)=\mathcal{H}^t_n$).
Each gap $G^t$ in $U^t_n$ is a hyperbolic convex set with barycenter $b^t_G$ using the Cayley-Klein
model of $U^t_n$. For each $t$ define $H^t_n(b^0_G)=b^t_G$ and extend $H_n$ over $G$
by taking the \lq\lq cone\rq\rq of $H_n$ over the boundary of $G$ and its barycenter $b^0_G$
(using the Cayley-Klein model). Note that if $G_i$ is a convergent sequence of gaps in $U_n$, then either
$\lim G_i$ is a point in the boundary of $U_n$, or a leaf $C$ in $\mathcal{H}_n$. In the
latter case the barycenters $b_i$ of $G_i$ converge to the barycenter of $C$.
It follows
that the isotopy $h$ can be extended to an isotopy $H_n$ of $Z\cup U_n$ for each $n$.

Let $H=\cup H_n$. Then $H:\sphere\times[0,1]\to\sphere$  is continuous by (\ref{smallcenter}).
Hence $H$ is the required extension of $h$.
\end{proof}

Theorem~\ref{mainsphere} shows that we can extend an isotopy $h$ of a planar continuum $Z$, starting at the identity,
to an isotopy $H:\sphere\times [0,1]\to\sphere$ of the sphere. Let $U$ denote the component
of $\sphere\setminus Z$ containing the point at infinity. By composing the isotopy $H$ by an isotopy $K$
of the sphere such that $K^0=id_{\sphere}$ and $K^t|_{\sphere\setminus U}=id_{\sphere\setminus U}$,
 and $K^t(H^t(\infty))=\infty$ for all $t\in[0,1]$ we
 obtain an isotopy which extends $h$ and fixes the point
at infinity. Hence the following theorem follows.

\begin{thm}\label{main}
Suppose that $h^t$ is an isotopy of a plane continuum $Z\subset \complex$
with $h^0=id|_Z$.
Then there exists an extension to an isotopy $H^t:\complex\to\complex$ such that $H^0=id$.
\end{thm}

\bibliographystyle{amsalpha}
\bibliography{/lex/references/refshort}

\providecommand{\bysame}{\leavevmode\hbox to3em{\hrulefill}\thinspace}
\providecommand{\MR}{\relax\ifhmode\unskip\space\fi MR }
\providecommand{\MRhref}[2]{%
  \href{http://www.ams.org/mathscinet-getitem?mr=#1}{#2}
}
\providecommand{\href}[2]{#2}

\end{document}